\documentclass[12pt]{elsarticle}
\usepackage[left=1in,right=1in, top=1.2in,bottom=1.2in]{geometry}
\usepackage{EJGTAart}
\usepackage{times}
\usepackage{amssymb,amsthm,latexsym,amsmath,epsfig,pgf}

\volume{{\bf 4} (1)}

\firstpage{1}

\runauth{
Constructing Arbitrarily Large Graphs with a Specified Number of Hamiltonian Cycles
\hspace{2ex} $\arrowvert$\hspace{2ex}
Michael Haythorpe}

\newtheorem{theorem}{Theorem}[section]
\newtheorem*{theorem A}{Theorem A}
\newtheorem*{theorem B}{N\"olker's Theorem}

\newtheorem{definition}{Definition}[section]

\theoremstyle{remark}

\theoremstyle{remark}

\begin{document}

\begin{frontmatter}
% Write your paper title here
\papertitle{Constructing Arbitrarily Large Graphs with a Specified Number of Hamiltonian Cycles}

%% use optional labels to link authors explicitly to addresses. The label may be more than one, use comma to separate

%\author[label1,label]{Edy Tri Baskoro${}^1$}

\author[label1]{Michael Haythorpe}
%\author[label2]{Joe Ryan}
%\author[label3]{Kiki A. Sugeng}

\address[label1]{\small School of Computer Science, Engineering and Mathematics,\\
Flinders University,\\
1284 South Road, Clovelly Park, SA 5042, Australia
%\address[label2]{School of Electrical Engineering and Computer Science,
%The University of Newcatle, Australia}
%\address[label3]{Departments of Mathematics,
%University of Indonesia,
%Depok - Indonesia

\vspace*{2.5ex}
 {\normalfont michael.haythorpe@flinders.edu.au}
 
}
\begin{abstract}
A constructive method is provided that outputs a directed graph which is named a broken crown graph, containing $5n-9$ vertices and $k$ Hamiltonian cycles for any choice of integers $n \geq k \geq 4$. The construction is not designed to be minimal in any sense, but rather to ensure that the graphs produced remain non-trivial instances of the Hamiltonian cycle problem even when $k$ is chosen to be much smaller than $n$.

\let\thefootnote\relax\footnotetext{Received: 29 August 2014,\quad Revised: 7 March 2016,\quad   Accepted: 11 March 2016.\\[3ex]
  }

\end{abstract}


\end{frontmatter}

\section{Introduction}

The Hamiltonian cycle problem (HCP) is a famous NP-complete problem in which one must determine whether a given graph contains a simple cycle traversing all vertices of the graph, or not. Such a simple cycle is called a Hamiltonian cycle (HC), and a graph containing at least one Hamiltonian cycle is said to be a Hamiltonian graph.

Typically, randomly generated graphs (such as Erd\H{o}s-R\'{e}nyi graphs), if connected, are Hamiltonian and contain many Hamiltonian cycles. Although HCP is an NP-complete problem, for these graphs it is often fairly easy for a sophisticated heuristic (e.g. see Concorde \cite{concorde}, Keld Helsgaun's LKH \cite{LKH} or Snakes-and-ladders Heuristic \cite{SLH}) to discover one of the multitude of Hamiltonian cycles through a clever search. However, some Hamiltonian graphs may contain only a small number of Hamiltonian cycles. Among the most famous of these is the infinite family of 3-regular graphs known as generalized Petersen graphs \cite{GP} GP$(n,2)$, which for $n = 3 \mod 6$ always contain 3 Hamiltonian cycles even though the graphs can grow arbitrarily large, containing $2n$ vertices. Even for dense graphs, it is possible that the number of Hamiltonian cycles is small, with Sheehan's infinite family of maximally uniquely-Hamiltonian graphs \cite{sheehan} (that is, graphs with exactly one Hamiltonian cycle and the maximum possible ratio of edges to vertices) being one such result.

However, there are situations where it may be desirable to be able to specify the number of Hamiltonian cycles desired in a constructed graph, without restricting the number of vertices too heavily. Such examples are often among the most taxing for HCP algorithms, and provide excellent instances for benchmarking. The generalized Petersen graphs GP$(n,2)$ for $n = 1 \mod 6$ contain exactly $n$ Hamiltonian cycles, but there is no control over the order of the graph, which is fixed at $2n$ vertices.

In this manuscript, a constructive procedure will be presented that, for any choice of integer $n \geq 4$, outputs a directed graph of order $5n-9$ containing $n$ Hamiltonian cycles (in a directed sense). For each Hamiltonian cycle, there will be two directed edges which are traversed by only that Hamiltonian cycle and none of the others. Therefore, the removal of either of these two directed edges eliminates that Hamiltonian cycle from the graph while preserving the rest. Then, if exactly $k$ Hamiltonian cycles are desired, one may simply remove $n-k$ directed edges to obtain such a graph. Obviously, this can be performed whenever $n \geq k$.

It is worth noting that the construction outlined in this manuscript is not the only such method for constructing graphs with a controlled number of Hamiltonian cycles, or even a minimal construction in any sense. Rather, the construction is designed in such a way that in addition to providing a reasonably small graph, the structural complexity of the graph itself is not diminished even when $k$ is much smaller than $n$. To see why this is a noteworthy feature, consider the well-known Wheel graph $W_n$ \cite{harary} which can be thought of simply as a cycle graph of length $n-1$ along with an additional vertex $v$ which is connected to all other vertices. It is easy to determine that this graph contains $n-1$ Hamiltonian cycles and each edge incident on $v$ is used in exactly two Hamiltonian cycles, so it is possible to control the number of Hamiltonian cycles to some degree by removing these edges. However, if most of the edges are removed to ensure only a small number of Hamiltonian cycles are present, the remaining graph is a trivial instance of HCP where almost every vertex is degree 2 and the few remaining Hamiltonian cycles are easy to discover. Such an example is displayed in Figure \ref{fig-wheel}, where the left graph is $W_{9}$ and the right graph is the modified version with only two Hamiltonian cycles, which are trivial to find.

\vspace*{1cm}\begin{figure}[h!]\begin{center}\includegraphics[scale=0.35]{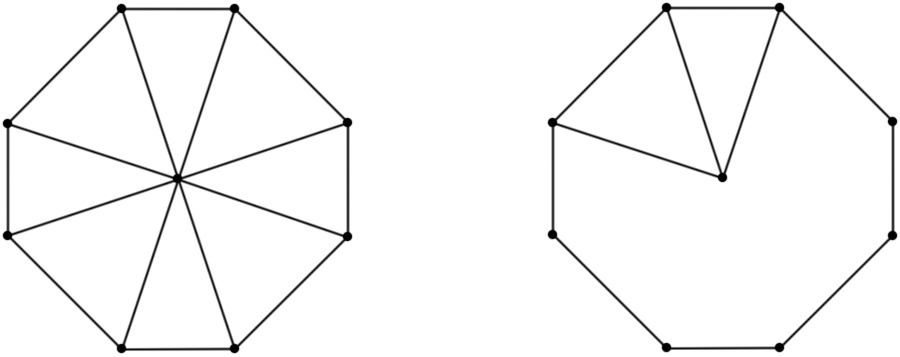}\caption{The Wheel Graph $W_{9}$ and the modified version with edges removed to ensure only 2 Hamiltonian cycles remain.\label{fig-wheel}}\end{center}\end{figure}

\section{Crown Subgraph}

In order to produce the desired graph, it is first necessary to introduce a parametrised family of subgraphs for each integer $n \geq 4.$

\begin{definition}The {\em Crown Subgraph} with parameter $n$ that is a natural number greater than or equal to 4, denoted by $\mathcal{C}_n$, is a directed subgraph containing $5n - 10$ vertices, and the following directed edges:

\begin{itemize}\item $(i,i+1)$ and $(i+1,i)$ for all $i = 1, \hdots, 5n - 11$,
\item $(1,5n-10)$ and $(5n-10,1),$
\item $(2n,2n-2)$ and $(5n-10,2),$
\item $(2n+3i,2n-2-2i)$, for all $i = 1, \hdots, \lceil \frac{n-4}{2} \rceil$,
\item $(5n-10-3i,2i+2)$, for all $i = 1, \hdots, \lfloor \frac{n-4}{2} \rfloor$.\end{itemize}\end{definition}

Suppose that the subgraph $\mathcal{C}_n$ is included inside a larger graph. Then there will be a set of incoming edges that go into $\mathcal{C}_n$ from other vertices in the larger graph, and a set of outgoing edges that depart from $\mathcal{C}_n$. Consider the situation where there are $n$ incoming edges and $n$ outgoing edges (each of which has a label from $1$ to $n$), which are incident on the following vertices:

\begin{itemize}\item Incoming edge labelled $i$ is incident to vertex $2i-1$ for all $i = 1, \hdots, n$,
\item Outgoing edge labelled $1$ is incident from vertex $5n-10$,
\item Outgoing edge labelled $2$ is incident from vertex $1$,
\item Outgoing edge labelled $n-1$ is incident from vertex $2n-1$,
\item Outgoing edge labelled $n$ is incident from vertex $2n$,
\item Outgoing edge labelled $i+2$ is incident from vertex $5n-9-3i$ for all $i = 1, \hdots, \lfloor \frac{n-4}{2} \rfloor$,
\item Outgoing edge labelled $n-1-i$ is incident from vertex $2n-1+3i$ for all $i = 1, \hdots, \lceil \frac{n-4}{2} \rceil$.\end{itemize}

The Crown Subgraph is illustrated in Figures \ref{fig-crown} and \ref{fig-crown2}, showing the subgraph before and after the addition of the incoming/outgoing edges. Because of the way the figure is displayed, vertices $1, \hdots, 2n-1$ are referred to as {\em top vertices} and the remaining vertices are referred to as {\em bottom vertices}. Notice that the incoming edges are all incident to the top vertices.

\vspace*{1cm}\begin{figure}[h!]\begin{center}\includegraphics[scale=0.3]{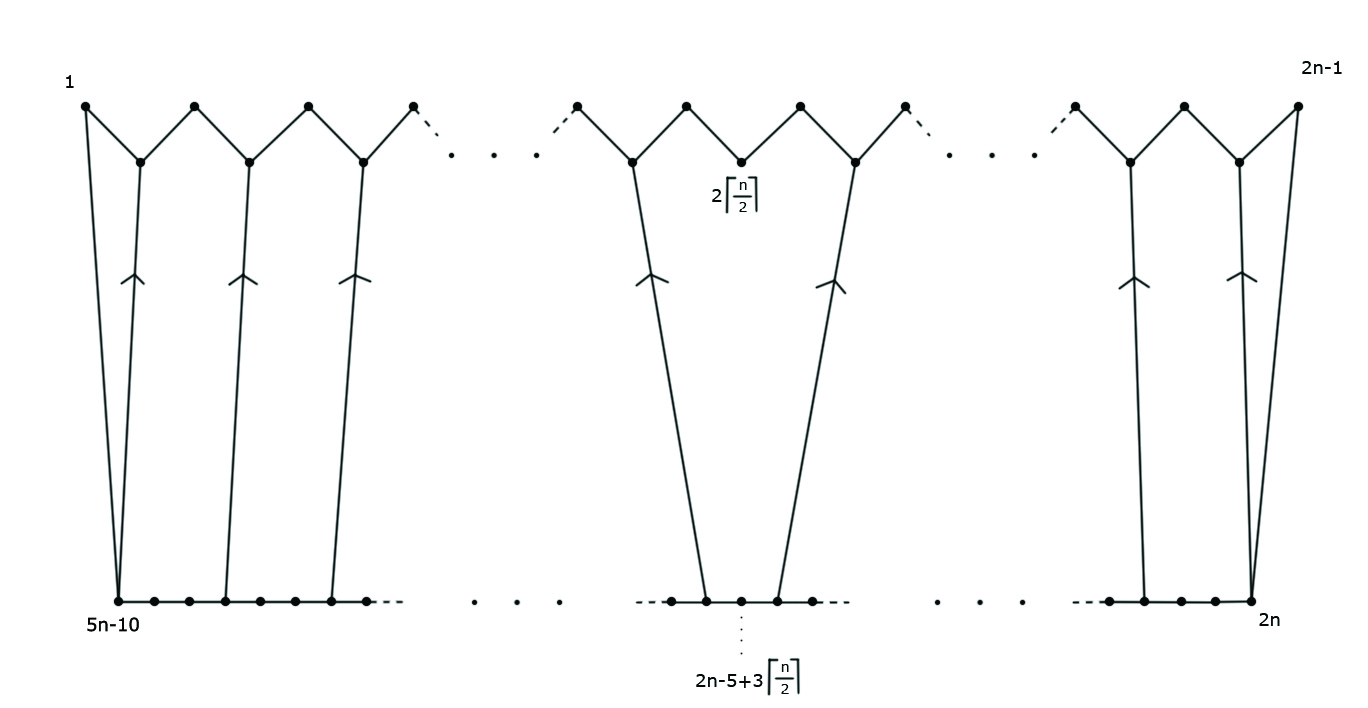}\caption{The Crown Subgraph with parameter $n$. The top-left vertex is labelled 1, and the vertex labellings increase in a clockwise fashion.\label{fig-crown}}\end{center}\end{figure}

\vspace*{1cm}\begin{figure}[h!]\begin{center}\includegraphics[scale=0.3]{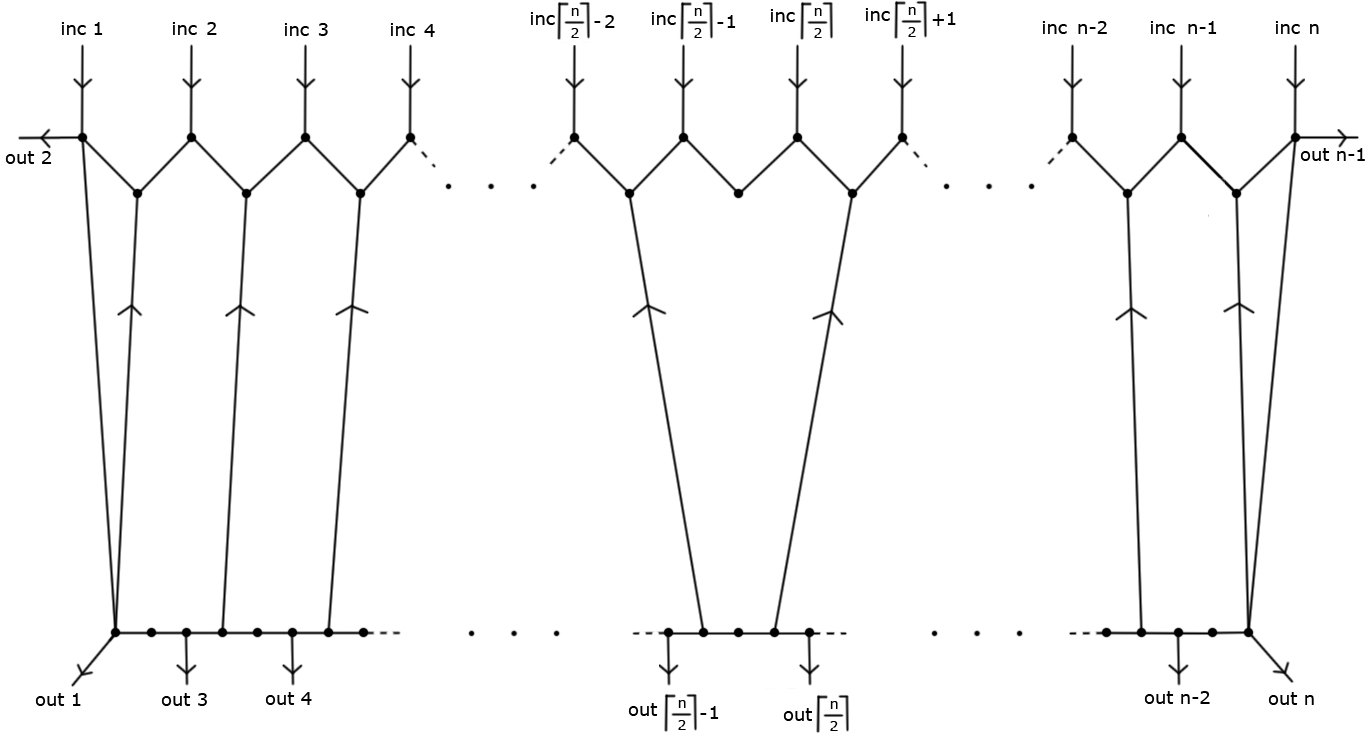}\caption{The Crown Subgraph with parameter $n$ and incoming and outgoing edges attached.\label{fig-crown2}}\end{center}\end{figure}

\begin{theorem}Any Hamiltonian cycle in a graph containing $\mathcal{C}_n$ must only traverse a single incoming edge and a single outgoing edge, and both edges must have the same label. Furthermore, there is only one path a Hamiltonian cycle may take between any given pair of incoming and outgoing edges.\label{thm-crown}\end{theorem}

\begin{proof}The boundary cases (ie incoming edges $1$, $2$, $n-1$ and $n$) must be considered separately from the other cases. Specifically, incoming edges 1 and $n$ will be considered together as the first case, and incoming edges 2 and $n-1$ will be considered together as the second case. Then the more general cases will be considered together as a third case. In each case, the proof will follow a similar pattern. Multiple possibilities will be examined, and all but one of them will lead to a contradiction in the sense that it will be impossible to avoid creating a short cycle. The remaining possibility will be to traverse the entirety of $\mathcal{C}_n$ and depart via the corresponding outgoing edge.

{\bf Case 1}: Suppose the HC enters $\mathcal{C}_n$ via incoming edge $1$. There are then three choices - the HC may either depart $\mathcal{C}_n$ immediately via outgoing edge $2$, travel right to vertex $2$, or down to vertex $5n-10$. If the HC departs $\mathcal{C}_n$ immediately, then $\mathcal{C}_n$ must be re-entered at a different time. During this alternative visit, vertex 2 needs to be visited, which can either occur by entering into one of the top vertices, and travelling left along the top vertices until vertex 2 is reached, or by first travelling to vertex $5n-10$ and then up to vertex 2. In the former case, a contradiction is reached because vertex 2 then cannot be departed without creating a short cycle. In the latter case, upon travelling up to vertex 2, the only option is to continue right along the top vertices. However, eventually the vertex which was used to re-enter $\mathcal{C}_n$ will be reached again, creating a short cycle. Since neither of the cases work, it is clear that the HC cannot immediately depart $\mathcal{C}_n$ after arriving via incoming edge 1.

Suppose instead that after arriving via incoming edge 1, the HC travels down to vertex $5n-10$. This vertex is adjacent to the degree 2 vertex $5n-9$ and so the HC must travel there immediately. However, then the same argument as above can be used to show that it is now impossible to visit vertex 2 without needing to create a short cycle. So this option also induces a contradiction. The only remaining option is to travel from vertex 1 straight to vertex 2. Then, the HC must continue along all of the top vertices. However, once it visits vertex $2n-1$, all of the incoming edges are incident to vertices which have already been visited. Therefore $\mathcal{C}_n$ cannot be re-entered, and so all vertices in $\mathcal{C}_n$ must be visited before departing. The only path left which visits the remaining vertices is to travel down to vertex $2n$, and go left along the bottom vertices until vertex $5n-10$ is reached, and outgoing edge $1$ is used to depart. This is the only valid path that may be used as part of a HC going through incoming and outgoing edges $1$.

Due to symmetry, an equivalent argument can be used to prove the theorem for incoming edge $n$.

{\bf Case 2}: Suppose the HC enters $\mathcal{C}_n$ via incoming edge $2$. There are then two choices - the HC may go left to vertex 2, or right to vertex 4. Suppose it goes left to vertex 2. By the same argument as in Case 1, the HC must then visit vertex 4 before it departs $\mathcal{C}_n$, since it will otherwise be impossible to re-enter and visit it later without creating a short cycle. Vertex 4 must be visited from one of the bottom vertices, or else a short cycle is created. Therefore, from vertex 3 the HC must travel through vertices 2, 1, $5n-10$, $5n-9$, $5n-8$, $5n-7$ and then up to vertex $4$. However, vertex $5n-7$ is adjacent to the degree 2 vertex $5n-6$, and so this path cannot be used in a HC without later creating a short cycle.

The only remaining alternative is to go right to vertex $4$, after entering via incoming edge $2$. Similarly to the previous argument, the HC must then visit vertex $2$ before it departs $\mathcal{C}_n$. The only valid way to do this is to travel right along all the top vertices, then down to vertex $2n$, and left along all the bottom vertices until vertex $5n-10$ is reached. At this stage the HC can either travel to vertices 1 or 2, but clearly travelling to vertex 1 means vertex 2 can't be visited without creating a short cycle, so the remaining option is to travel from $5n-10$ to vertex 2, then on to vertex 1, and to depart via outgoing edge $2$. This is the only valid path that may be used as part of a HC going through incoming and outgoing edges $2$.

Due to symmetry, an equivalent argument can be used to prove the theorem for incoming edge $n-1$.

{\bf Case 3}: Suppose the HC enters $\mathcal{C}_n$ via incoming edge $i+2$ for some $i \in \left[1, \hdots, \lfloor \frac{n-4}{2} \rfloor\right]$, arriving at vertex $2i+3$. The HC can then either go left to vertex $2i+2$ or right to vertex $2i+4$. However, using an equivalent argument to that in Case 2, the HC cannot go left first, or else it becomes impossible to later visit vertex $2i+4$ without creating a short cycle. Therefore, the HC goes right to vertex $2i+4$ first, and then as before, visit vertex $2i+2$ before departing $\mathcal{C}_n$. Again, using the same arguments as previously, vertex $2i+2$ must be visited from one of the bottom vertices, of which the only choice is vertex $5n-10-3i$. Therefore the HC to this point enters at vertex $2i+3$, travels right along the top vertices until vertex $2n-1$, then travels down to vertex $2n$, goes left along the bottom vertices until vertex $5n-10-3i$, and then travels up to vertex $2i+2$. It must then continue left along the top vertices until vertex $1$. It cannot immediately depart here because all vertices incident to incoming edges have now been visited, and so the remaining vertices in $\mathcal{C}_n$ must be visited before departing. The only remaining option is to then travel down to vertex $5n-10$, and go right along the bottom vertices until vertex $5n-9-3i$ is reached, at which time the HC departs via outgoing edge $i$. This is the only valid path that may be used as part of a HC going through incoming and outgoing edges $i$.

Due to symmetric, an equivalent argument can be used to prove the theorem for outgoing edge $n-i-1$ for any $i = 1, \hdots, \lceil \frac{n-4}{2} \rceil$.

Since all incoming edges have now been considered, the proof is concluded.\end{proof}

\section{Broken Crown Graph Construction}

Consider a graph constructed by taking a copy of $\mathcal{C}_n$ and attaching each of the incoming and outgoing edges to a single vertex $v$. From Theorem \ref{thm-crown} it may be immediately concluded that this graph contains exactly $n$ Hamiltonian cycles, where the $i$-th Hamiltonian cycle travels from $v$ along incoming edge $i$ to vertex $2i-1$, then through $\mathcal{C}_n$ in a unique way before finally travelling along outgoing edge $i$ and returning to $v$.

By removing any of the incoming or outgoing edges, precisely one of the Hamiltonian cycles is eliminated from the graph. Then, if a particular number of Hamiltonian cycles is required (bounded above by $n$), one can simply remove the desired number of incoming or outgoing edges. Depending on the desired properties of the resulting graph, it may be preferable to remove both the incoming and outgoing edge corresponding to each eliminated Hamiltonian cycle, or just one of them.

\begin{definition}A {\em broken crown} graph $\mathcal{B}_{n,k}$ is any graph constructed as above containing, $\mathcal{C}_n$ and modified so that only $k$ Hamiltonian cycles remain.\label{def-broken}\end{definition}

One beneficial property of broken crown graphs is that any broken crown graphs $\mathcal{B}_{n,k}$ and $\mathcal{B}_{n,j}$ will be structurally quite similar graphs for any $k$ and $j$. In terms of their use as benchmark instances, retaining a similar structure while having control over the number of Hamiltonian cycles allows a purer test of how the number of Hamiltonian cycles impacts on a given HCP algorithm. Even if $k$ is chosen to be much smaller than $n$, $\mathcal{B}_{n,k}$ is structurally no simpler than $\mathcal{B}_{n,n}$ as an instance of HCP.

One potential concern is that the vast majority of HCP algorithms are not designed for directed graphs such as $\mathcal{B}_{n,k}$. This can be remedied through the use of the well known conversion from directed HCP to undirected HCP \cite{karp}. This is done by replacing every vertex $i$ with three vertices $j_1$, $j_2$, $j_3$ and edges $(j_1,j_2)$, $(j_2,j_3)$. Then for every edge going into vertex $i$ in the original graph, a corresponding edge in the undirected graph is incident to vertex $j_1$. Likewise, for every edge departing vertex $i$ in the original graph, a corresponding edge in the undirected graph is incident to vertex $j_3$. This conversion is good as the set of Hamiltonian cycles in the original graph has a 1-to-1 correspondence with the set of Hamiltonian cycles in the undirected graph. The final undirected graph will contain $15n - 27$ vertices, and between $22n + k - 40$ and $21n + 2k - 40$ edges (depending on whether both incoming and outgoing edges, or just one, was removed).

If some slight variation in the number of vertices is desired (ie if it not desirable for the undirected graph to contain $3 \mod 15$ vertices) it is possible to replace vertex $v$ with any other subgraph containing a unique Hamiltonian path, with the incoming edges all incident to the starting vertex in the Hamiltonian path, and the outgoing edges all incident from the finishing vertex in the Hamiltonian path. For example, $v$ could be replaced by any Sheehan graph of any desired size by simply removing any edge $(a,b)$ which is in the unique Hamiltonian cycle, and making all incoming edges incident to $a$ and all outgoing edges incident from $b$.

Alternatively, every time an outgoing edge is removed (with the exception of outgoing edges 2 and $n-1$) in the construction of a broken crown graph, one of the bottom vertices becomes a degree 2 vertex. Since it will always be adjacent to another degree 2 vertex, it is possible to contract these two into a single vertex. This can be repeated up to $n-k$ times (once for each removed outgoing edge) to have further control over the final size of the graph.

This manuscript is concluded with a visualisation of a broken crown graph $\mathcal{B}_{11,6}$, displayed in Figure \ref{fig-broken}. In this example, five outgoing edges were removed and all incoming edges were retained. The five enlarged vertices in Figure \ref{fig-broken} correspond to the five outgoing edges that were removed. The four enlarged bottom vertices could all be contracted as described above if desired.

\vspace*{1cm}\begin{figure}[h!]\begin{center}\includegraphics[scale=0.3]{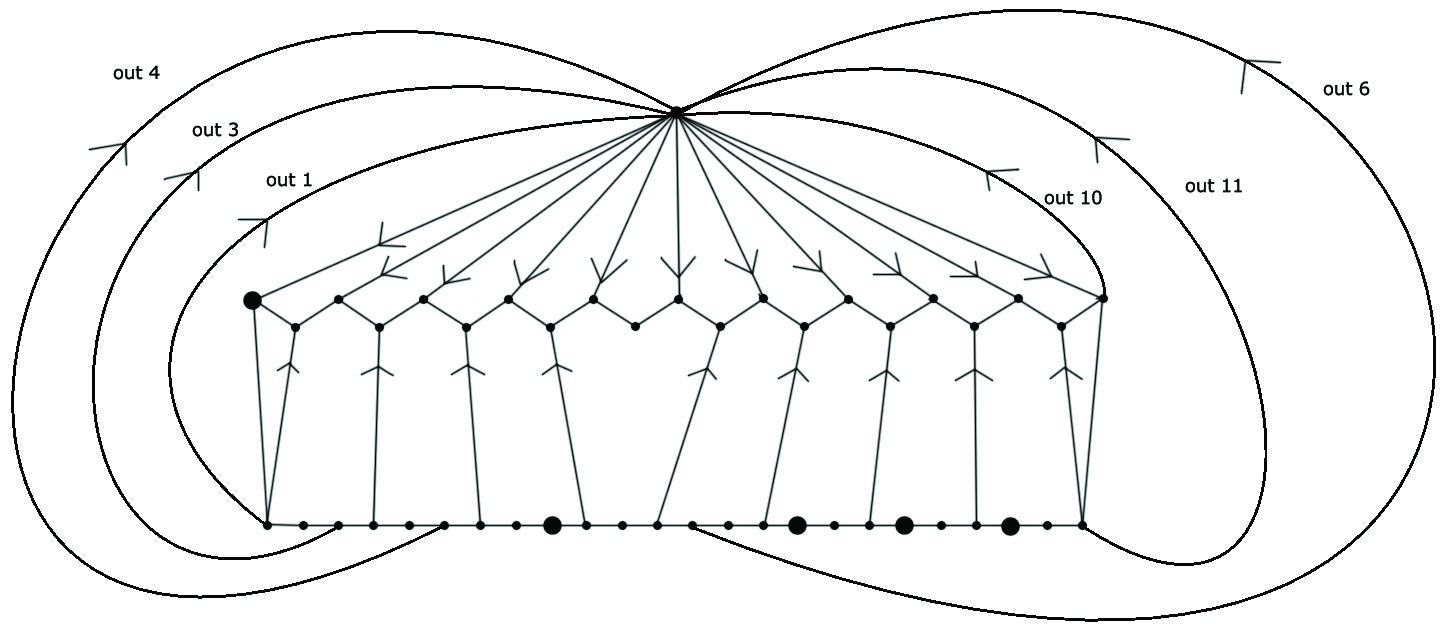}\caption{An example of a Broken Crown graph $\mathcal{B}_{11,6}$ constructed by removing outgoing edges 2, 5, 7, 8 and 9. The 11 incoming edges are labelled in order from left to right.}\label{fig-broken}\end{center}\end{figure}

\section*{Acknowledgement}
The research described in this manuscript was supported by a Research Agreement with DST Group, Australia.

\end{document}